\documentclass{amsart}

\usepackage[dvips]{graphicx}

\def\R{{\mathbb R}}
\def\N{{\mathbb N}}

\def\Z{{\mathbb Z}}
\def\1{{1\!\!\!1}}

\def\P{{\mathbb P}}

\def\cal{\mathcal}

\def\ol{\overline}

\def\dist{{\rm{dist}}}

\def\eps{\varepsilon}
\def\a{{\alpha}}

\newcommand{\be}{\begin{equation}}
\newcommand{\ee}{\end{equation}}
\numberwithin{equation}{section}

\newtheorem{theorem}{Theorem}
\newtheorem{prop}{Proposition}[section]

\newtheorem{defi}{Definition}[section]
\newtheorem{lemma}{Lemma}[section]

\title[The $t$-Martin Boundary ]{The  $t$-Martin  boundary of reflected random walks on a half-space}
\author{Irina Ignatiouk-Robert}
\address{
{Universit\'e de Cergy-Pontoise,}
{D\'epartement de math\'ematiques,}
{2, Avenue Adolphe Chauvin,}
{95302 Cergy-Pontoise Cedex,}
{France}}
\date{\today}
\email{Irina.Ignatiouk@math.u-cergy.fr}

\begin{document}
\begin{abstract} The $t$-Martin boundary of a random walk on a half-space with reflected
  boundary conditions is identified. It is shown in particular that  the   $t$-Martin boundary of such a
  random walk is not stable in the following sense : for different values of $t$,   the
  $t$-Martin compactifications are not homeomorphic to each other.   
\end{abstract}
\maketitle

\section{Introduction}
Before  formulating our results we recall  the definition and the properties of $t$-Martin compactification. 

Let $P=(p(x,x'), \; x,x'\in E)$ be a transition kernel of a time-homogeneous, irreducible Markov
chains $Z=(Z(t))$  on a countable, discrete state spaces $E$. Then by irreducibility, for any $t>0$, the series 
\be\label{e1-2}
G_t(z,z') ~\dot=~ \sum_{n=0}^\infty t^{-n} \P_x(Z(n) ~=~ z')
\ee
either converge or diverge simultaneously for all $z,z'\in E$ (see~Seneta~\cite{Seneta}).
\begin{defi} 
The infimum $\rho(P)$ of the $t>0$ for which the series \eqref{e1-2} converge is equal to 
\be\label{e1-1}
\rho(P) ~=~ \limsup_{n\to\infty} \left(\P_x(Z(n)=x')\right)^{1/n},
\ee 
it is called the convergence norm of the transition kernel $P$. 
\begin{enumerate}
\item For $t > 0$, a positive function $f : E\to\R_+$ is
said to be $t$-harmonic (resp. $t$-superharmonic) for $P$ if it satisfies the equality
$Pf = tf$ (resp.  $Pf \leq tf$). A $t$-harmonic function is therefore an  
eigenvectors of the transition operator $P$ with respect to the eigenvalue $t$.   For
$t=1$, the $t$-harmonic functions are called harmonic.   
\item A $t$-harmonic function $f>0$ is said to be minimal if for any  $t$-harmonic
  function $\tilde{f}>0$ the inequality $\tilde{f}\leq f$ implies the equality $\tilde{f}
  = c f$ with some $c>0$. 
\end{enumerate}
\end{defi}
For $t > 0$, the set of $t$-superharmonic functions of an irreducible Markov
kernel $P$ on a countable state space $E$ is nonvoid only if $t\geq\rho(P)$,
see~Pruitt~\cite{Pruitt} or Seneta~\cite{Seneta}. 

\begin{defi} 
 The $t$-Martin kernel $K_t(x,x')$ of the transition kernel $P$ is defined by 
\be\label{e1-4}
K_t(x,x') ~=~ G_t(x,x_n)/G_t(x_0,x_n)
\ee
where $x_0$ is a reference point  in $E$.

A sequence of points $x_n\in E$ is said to converge to a point of the $t$-Martin
  boundary $\partial_{t,M}(E)$ of the set $E$ defined by the transition kernel $P$ if  for any finite subset
  $V\subset E$ there is $n_V$ such that $x_n\not\in V$ for all $n > n_V$ and the sequence of functions
  $K_t(\cdot, x_n)$ converges point-wise on 
  $E$. 
\end{defi}
The $t$-Martin  compactification $E_{t,M}$ is therefore the unique  smallest compactification of
the  set $E$ for  which the  $t$-Martin kernels $K_t(z,\cdot)$  extend continuously.
\begin{defi} The $t$-Martin compactification is said to be stable if it does
  not depend on $t$ for $t>\rho(P)$, i.e. if for any sequence of points $x_n\in E$ that leaves the finite
subsets of $E$, the convergence to a point of the $t$-Martin boundary for some $t >
\rho(P)$ implies the convergence to a point of the $t$-Martin boundary for all $t >
\rho(P)$.  
\end{defi}
In the case $t=1$ and with a transient transition kernel $P$, the $t$-Martin compactification is the
classical Martin compactification, introduced first for Brownian motion by
Martin~\cite{Martin}. For countable Markov chains with discrete time, the abstract construction
of the Martin compactification was given by Doob~\cite{Doob} and Hunt~\cite{Hunt}. The
main general results in this domain are the following~:

The minimal
Martin boundary $\partial_{1,m}(E)$ is the set of all those
$\gamma\in\partial_{1,M}(E)$ for which the function $K_1(\cdot, \gamma)$ is 
minimal harmonic. By the Poisson-Martin
representation  theorem, for  every  non-negative  $1$-harmonic function  $h$  there exists  a
unique positive Borel measure $\nu$ on $\partial_{1,m}(E)$ such that
\[
h(z) = \int_{\partial_{t,m} E_M} K_1(z,\eta) \, d\nu(\eta) 
\]
By Convergence theorem, the sequence
$(Z(n))$ converges $\P_z$ almost surely for every initial state $z\in E$ to a $\partial_{1,m}(E)$ valued
random variable. The Martin boundary provides therefore all non-negative $1$-harmonic functions and 
describes the asymptotic behavior of the transient Markov chain $(Z(n))$. See
Woess~\cite{Woess}).

In general it is a non-trivial problem  to determine Martin boundary of a given class of
Markov chains. The $t$-Martin boundary plays an important role to determine the Martin
boundary of several products of transition kernels.

\begin{enumerate}
\item To  identify the Martin boundary of the direct  product of two independent
transient   Markov  chains   $(X(n))$  and   $(Y(n))$,  i.e.    the  Martin   boundary  of
$Z(n)=(X(n),Y(n))$, the  determination of  the Martin boundary  of each of  the components
$(X(n))$ and $(Y(n))$ is far from being sufficient.  Molchanov~\cite{Molchanov} has shown that
for  strongly aperiodic  irreducible Markov  chains $(X(n))$  and $(Y(n))$,  every minimal
harmonic function $h$ of the couple  $Z(n)=(X(n),Y(n))$ is of the form $h(x,y) = f(x)g(y)$
where $f$  is a $t$-harmonic function  of $(X(n))$ and  $g$ is a $s$-harmonic  function of
$(Y(n))$ with  some $t>0$ and  $s>0$ satisfying the  equality $t s =  1$.  

\item In the  case of Cartesian product  of Markov  chains, i.e.  by  considering a convex
  combination $Q=aP+(1-a)P'$, $0<a<1$,     of    the     corresponding    transition
  matrices,     Picardello    and Woess~\cite{Picardello-Woess} has shown  that the
  minimal harmonic functions of the transition matrix $Q$ have a similar product form but with $t>0$ and
  $s>0$ satisfying the equality $at + (1-a)s = 1$. In this paper some of the results on
  the topology of the Martin boundary are
  obtained under the assumption that the   $t$-Martin  boundaries of  the components
  $(X(n))$ and  $(Y(n))$ are  stable in the   above sense. 
\end{enumerate}
This stability property is  an important ingredient for the identification of the
Martin boundary of the product of Markov chains in general. The assumption on stability seems to be non-restrictive in the
case of  (spatially) homogeneous Markov processes, see  Woess~\cite{Woess}, Picardello and
Woess\cite{Picardello-Woess:2}).    These   previous   works  suggest   in
particular the natural conjecture that the $t$-Martin compactification should be stable in
general.  The  purpose of this  paper is to  show that this is  not true.  The
$t$-Martin compactification  of a  random walk on  a half-space $\Z^{d-1}\times\N$  with a
reflected boundary conditions on  the hyper-plane $\Z^{d-1}\times\{0\}$ is identified.   Our  results  show  in
particular that  the $t$-Martin compactification for such a random walk is not stable.

\section{Main results} We consider a random walk $Z(n) = (X(n),Y(n))$ on $\Z^{d-1}\times\N$ with transition probabilities 
\[
p(z,z') ~=~ \begin{cases} \mu(z'-z) &\text{for ~$z=(x,y),z'\in\Z^{d-1}\times\N$ with $y>0$,}\\
\mu_0(z'-z) &\text{for $z=(x,y),z'\in\Z^{d-1}\times\N$ with $y=0$} 
\end{cases}
\]
where $\mu$ and $\mu_0$ are two different positive measures on $\Z^d$ with $0 < \mu(\Z^d)\leq
1$ and $0 < \mu_0(\Z^d) \leq 1$. The random walk $Z(n) = (X(n),Y(n))$ can be therefore
substochastic if either $\mu(\Z^d) <  1$ or $\mu_0(\Z^d) < 1$.

Throughout this
paper we denote by $\N$ the set of all non-negative integers~: $\N=\{0,1,2,\ldots\}$ and
we let $\N^*=\N\setminus\{0\}$. The assumptions we need on the Markov process
$(Z(t))$ are the following. 

\medskip
\noindent 
\begin{enumerate}
\item[(H0)] {\em $\mu(z)=0$ for $z=(x,y)\in\Z^{d-1}\times\Z$ with $y<-1$ and $\mu_0(z) = 0$ for
$z=(x,y)\in\Z^{d-1}\times\Z$ with $y< 0$.} 
\item[(H1)] {\em The Markov process $Z(t)$ is irreducible on
  $\Z^{d-1}\times\N$.} 
\item[(H2)] {\em The homogeneous random walk $S(t)$ on $\Z^d$ 
having transition probabilities  
$p_S(z,z')=\mu(z'-z)$ is irreducible on $\Z^d$}. 
\item[(H3)]{\em  The jump generating functions 
\be\label{e1-6}
\varphi(a) ~=~ \sum_{z\in\Z^d} \mu(z) e^{a\cdot z} \quad \text{ and } \quad \varphi_0(a) ~=~
\sum_{z\in\Z^d} \mu_0(z) e^{a\cdot z} 
\ee
are finite everywhere on $\R^d$.} 
\item[(H4)]{\em The last coordinate of $S(t)$ is an aperiodic random walk on $\Z$ .}
\end{enumerate}

Our first preliminary result identifies the convergence rate $\rho(P)$ of the transition
kernel $P=(p(z,z'), \; z,z'\in\Z^{d-1}\times\N)$. 
\begin{prop}\label{pr1} Under the hypotheses (H0)-(H3), 
\be\label{e1-7}
\rho(P) ~=~ \inf_{a\in\R^d} \max\{\varphi(a),\varphi_0(a)\}.  
\ee
\end{prop} 
This is a consequence of the large deviation  principle for sample
  paths of the scaled processes $Z^\eps(t) ~\dot=~ \eps 
Z(t/\eps)$ obtained in \cite{D-E-W,D-E,Ignatiouk:02,Ignatiouk:04} (for the related
results see also \cite{D-B,
  D-E-2,I-M-S,S-W}). 
The proof of this proposition is given in Section~\ref{proof-of-prop1}.

\medskip

Remark that under the assumptions (H0)-(H3), for any $t>0$, the sets 
\be\label{e1-8}
D^t ~\dot=~ \{a\in \R^d :
\varphi(a) \leq t\} \quad \text{and} \quad D^t_0 ~\dot=~ \{a\in \R^d :
\varphi_0(a) \leq t\} 
\ee 
are convex and the set $D^t$ is moreover compact. 
We denote by $\partial D^t$ the boundary of $D^t$ we let 
\[
\partial_0 D^t ~\dot=~ \{a\in\partial D^ :~ \nabla\varphi(a) \in
\R^{d-1}\times\{0\}\},
\] 
\[
\partial_+ D^t ~\dot=~ \{a\in\partial D^t :~ \nabla\varphi(a) \in
\R^{d-1}\times[0,+\infty[\}
\] 
and  
\[
\partial_- D^t ~\dot=~ \{a\in\partial D^t :~ \nabla\varphi(a) \in
\R^{d-1}\times ]-\infty,0]\}. 
\]
 For   
$a\in D^t$,   the unique
point on the boundary $\partial_- D^t$  which has the same first $(d-1)$ coordinates as the
point $a$ is denoted by $\ol{a}^t$,  
\be\label{e1-9}
\hat{D}^t ~\dot=~ \{a\in D^t : \varphi_0(\ol{a}^t)\leq t\} \quad \text{ and } \quad
\Gamma_+^t ~\dot=~ \partial_+D^t \cap\hat{D}^t.  
\ee
 Remark that $\partial_0 D^t = \partial_+ D^t\cap \partial_-D^t$ and for $a\in\partial_+
D^t$,  one has $a=\ol{a}^t$ if and only if $a\in\partial_0 D^t$. 
Moreover, under the hypotheses~(H0)-(H1), for any $a\in D^t$,
\[
\varphi_0(\ol{a}^t) \leq \varphi_0(a) 
\]
because the function $a\to\varphi_0(a)$ is increasing with respect to the last coordinate
of $a\in\R^d$. This inequality implies  another useful
representation of the set $\hat{D}^t$ : 
\[\text{ $a=(\a,\beta)\in\hat{D}^t$ if
and only if $a\in D^t$ and $a'=(\a,\beta')\in D^t\cap D_0^t$ for some  $\beta' \in\R$ } 
\]
or equivalently,
\be\label{e1-10}
\hat{D}^t ~=~ (\Theta^t\times\R) \cap D^t 
\ee
where 
\be\label{e1-11}
\Theta^t ~\dot=~ \{ \a \in\R^{d-1} :~ \inf_{\beta\in\R} \max\{\varphi(\a,\beta),
\varphi_0(\a,\beta)\} \leq t\}.
\ee 
The set $\Theta^t\times\{0\}$ is therefore the orthogonal projection of the set $D^t\cap D_0^t$ onto the hyper-plane
$\R^{d-1}\times\{0\}$ and by Proposition~\ref{pr1},
\be\label{e1-12}
\rho(P) ~=~ \inf\{t > 0 :~ D^t\cap D^t_0 \not=\emptyset\} ~=~ \inf\{ t > 0 : \Theta^t
\not=\emptyset\}. 
\ee
For $t > \rho(P)$ and 
$a\in \hat{D}^t$, we denote by $V_t(a)$ the normal cone to the set 
$\hat{D}^t$ at the point $a$ and for $a\in\Gamma_+^t ~\dot=~ \hat{D}^t\cap\partial_+D^t=
(\Theta^t\times\R) \cap \partial_+D^t$ we define the function 
$h_{a,t}$ on $\Z^{d-1}\times\N$ by letting 
\be\label{e1-13}
h_{a,t}(z) ~=~ \begin{cases} \displaystyle{\exp(a\cdot z)  -
  ~\frac{t - \varphi_0(a)}{t -
  \varphi_0(\ol{a}^t)} \, \exp(\ol{a}^t\cdot z})  &\text{ \!if  
 $\; a\not\in\partial_0 D^t \; $ and  $\; \varphi_0(\ol{a}^t) < t$,}\\ 
\\
\displaystyle{y \exp(a\cdot z) +  
 ~\frac{ \frac{\partial}{\partial\beta}\varphi_0(a)}{(t -
  \varphi_0(a))} \!\exp(a\cdot z)} &\text{ \!if  $a=\ol{a}^t\in\partial_0 D^t$ and $\varphi_0(a) < t$,}\\
\\
\exp(\ol{a}^t\cdot z) &\text{ \!if $\; \varphi_0(\ol{a}^t) = t\; $}
\end{cases}
\ee
 where $\frac{\partial}{\partial\beta}\varphi_0(a)$ denotes the partial derivative of the
function $a\to \varphi_0(a)$ with respect to the last coordinate $\beta\in\R$ of $a=(\a,\beta)$.   

The following lemma gives an explicit representation of the normal cone $V_t(a)$. 
\begin{lemma}\label{lem2} Under the hypotheses (H0)-(H3), for any $t>\rho(P)$ and $a\in \Gamma_+^t$, 
\be\label{e1-14}
V_t(a) ~=~ \begin{cases} \bigl\{c\nabla\varphi(a) : c\geq 0\bigr\} &\text{if
    either 
    $\varphi_0(\ol{a}^t) < t$} \\
&\text{or $a=\ol{a}^t\in\partial_0 D^t$,}\\
\\
\bigl\{c_1\nabla\varphi(a) + c_2(\nabla\varphi_0(\ol{a}^t) + \kappa_{a}
\nabla\varphi(\ol{a}^t)) :~ c_i\geq 0\bigr\} &\text{if
    $\varphi_0(\ol{a}^t) = t$ }\\ &\text{and $a\not\in\partial_0 D^t$}\\ 
\end{cases}
\ee
where 
\[
\kappa_{a} = - \left.\frac{\partial\varphi_0(\a,\beta)}{\partial\beta}
\left(\frac{\partial\varphi(\a,\beta)}{\partial\beta}\right)^{-1}\right|_{(\a,\beta) = \ol{a}^t} 
\]
\end{lemma}
\begin{proof} Recall that for any $t >  \inf_a \max\{\varphi(a,\varphi_0(a)\}$, the
  set $\Theta^t\times\{0\}$ is the orthogonal  projection of the convex set $D^t\cap
  D^t_0$ onto the hyperplane $\R^{d-1}\times\{0\}$. This proves that the set $\Theta^t$ is
  convex itself. Moreover, for any $t > \inf_a \max\{\varphi(a,\varphi_0(a)\}$, the
  set $D^t\cap D^t_0$ has a non-empty interior. Since $D^t\cap D^t_0\subset \hat{D}^t$
  from this it follows that for any $t > \inf_a \max\{\varphi(a,\varphi_0(a)\}$, 
  set $\hat{D}^t = (\Theta^t\times\R)\cap D^t$  has also a non-empty interior and
  consequently, by Corollary~23.8.1 of Rockafellar~\cite{R}, 
\be\label{e1-14p}
V_t(a) ~=~ V_{\Theta^t\times\R}(a) + V_{D^t}(a), \quad \forall a\in\hat{D}^t, 
\ee
where $V_{\Theta^t\times\R}(a)$ denotes the normal cone to the set $\Theta^t\times\R$ at
the point $a$ and $V_{D^t}(a)$ is the normal cone to the set $D^t$ at $a$. Since under the
hypotheses of our lemma, 
\be\label{e1-14pp}
V_{D^t}(a) ~=~ \bigl\{ c\nabla\varphi(a) :~ c\geq 0\bigr\}, \quad \forall a\in\partial D^t
\ee
from this it follows that 
\[
V_t(a) ~=~ V_{D^t}(a) ~=~ \bigl\{c\nabla\varphi(a) : c\geq 0\bigr\}  
\]
whenever the point $a\in \Gamma_+^t$ belongs to the interior of the set
$\Theta^t\times\R$, i.e. when $\varphi_0(\ol{a}^t) < t$. The first equality of
\eqref{e1-14} is therefore verified.  Suppose now that the point $a\in\Gamma_+^t$
belongs to the boundary of the set $\Theta^t\times\R$, i.e. either
$a=\ol{a}^t\in\partial_0 D^t$ or $\varphi_0(\ol{a}^t) = t$.  Then 
\[
V_{\Theta^t\times\R}(a) ~=~ V_{D^t\cap D^t_0}(\hat{a}^t) \cap (\R^{d-1}\times\{0\})
\]
because the set $\Theta^t\times\{0\}$ is the orthogonal projection of $D^t\cap D^t_0$ onto
$\R^{d-1}\times\{0\}$. $V_{D^t\cap D^t_0}(\hat{a}^t)$ denotes here the normal cone to the set $D^t\cap D^t_0$ at the point
$\hat{a}^t$. Using therefore again Corollary~23.8.1 of Rockafellar~\cite{R}, we obtain 
\[
V_{\Theta^t\times\R}(a) ~=~ \left(V_{D^t}(\ol{a}^t) + V_{D^t_0}(\ol{a}^t)\right) \cap (\R^{d-1}\times\{0\})
\] 
where 
\[
V_{D^t_0}(\ol{a}^t) ~=~ \begin{cases}\bigl\{ c\nabla\varphi_0(\ol{a}^t) :~ c\geq 0\bigr\}
  &\text{ if $\varphi_0(\ol{a}^t) = t$,}\\
\{ 0\} &\text{ if $\varphi_0(\ol{a}^t) < t$,}
\end{cases}
\]
is the normal cone to the set $D^t_0$ at the point $\ol{a}^t$. Since the function
$\varphi_0$ is increasing with respect to the last variable, the last coordinate of
$\nabla\varphi_0(\ol{a}^t)$ is strictly positive and consequently, the last relations
combined with \eqref{e1-14p} and \eqref{e1-14pp} prove the second equality of \eqref{e1-14}. 
\end{proof}

The main result of our paper is the following theorem. As above, we denote by $K_t(z,z')$
the $t$-Martin kernel of the Markov process $(Z(n))$ with a 
reference point $z_0\in\Z^{d-1}\times\N$ and 
\[
G_t(z,z') ~\dot=~ \sum_{n=0}^\infty t^{-n} \P_x(Z(n) ~=~ z'). 
\]

\begin{theorem}\label{th1} Under the hypotheses (H0)-(H4), for any $t>\rho(P)$, the following assertions hold~: 
\begin{itemize}
\item[(i)] for any unit vector $q\in\R^{d-1}\times[0,+\infty[$ there
    exists a unique $a =\hat{a}_t(q)\in \Gamma_+^t$  such that $q\in
    V_t(\hat{a}_t(q))$, 
\item[(ii)] for any
  $a\in\hat{D}^t\cap\partial_+ D^t$ and any sequence of points
  $z_n\in\Z^{d-1}\times\N$, 
\be\label{e1-15}
\lim_{n\to\infty} K_t(z,z_n) = h_{a,t}(z)/h_{a,t}(z_0),\quad \quad \quad \forall \;
z\in\Z^{d-1}\times\N 
\ee 
whenever $\lim_{n\to\infty} |z_n|=\infty$ and $\lim_{n\to\infty}
  \dist(V_t(a),z_n/|z_n|)=0$.
\end{itemize}
\end{theorem}
\noindent
Assertion (ii) proves  that a sequence  $z_n\in\Z^{d-1}\times\N$ with
$\lim_{n\to\infty}|z_n| = \infty$, converges to a point on the $t$-Martin boundary  if and only if 
\[
\lim_{n\to\infty} ~\dist\left(V_t(a),z_n/|z_n|\right)=
  0 
\]
for some $a\in\hat{D}\cap\partial_+ D$. The $t$-Martin compactification is therefore
stable if and only if $V_t(\hat{a}_t(q)) = V_s(\hat{a}_s(q))$ for any unit vector
$q\in\R^{d-1}\times[0,+\infty[$ and all $t > s > \rho(P)$. 

\medskip

Before proving our results, Theorem~\ref{th1} is illustrated on the example, it is shown that under quite general assumptions, the $t$-Martin
compactification of a random walk on a half-plane
$\Z\times\N$  is unstable.  This is a subject 
of Section~\ref{example}. In Section~\ref{proof-of-prop1}, we prove  Proposition~\ref{pr1}. 
Section~\ref{proof-of-theorem} is devoted to the proof of Theorem~\ref{th1}. 

\section{Example}\label{example}
Recall that under the hypotheses (H0)-(H3), by Proposition~\ref{pr1}, the convergence norm of our
transition kernel $P$ is given by 
\[
\rho(P) = \inf_{a\in\R^2}
\max\{\varphi(a), \varphi_0(a)\}.
\]
In this section, we consider a particular case when  $d=2$ and 
\be\label{e2-1}
\inf_{a\in\R^2} \max\{\varphi(a), \varphi_0(a)\} ~>~ \inf_{a\in\R^2} \varphi(a). 
\ee
Then the minimum of function $\max\{\varphi(a), \varphi_0(a)\}$ over $a\in\R^2$ is
achieved at some point $a^*=(\a^*,\beta^*)$ where 
\[
\left.\frac{\partial}{\partial\beta}\varphi(\a,\beta)\right|_{(\a,\beta) = a^*} ~\leq~ 0 \quad
\text{ and } \quad \left.\frac{\partial}{\partial\beta}\varphi_0(\a,\beta)\right|_{(\a,\beta) = a^*} ~>~ 0.
\]
The second inequality holds here because the function $\varphi_0(\a,\beta)$ is increasing with
respect to the second variable $\beta$, and to prove the first inequality it is
sufficient to notice that otherwise, there is another point $a=(\a,\beta)$ with $\a=\a^*$
and $\beta < \beta^*$ for which $
\max\{\varphi(a), \varphi_0(a)\} ~<~ \max\{\varphi(a^*), \varphi_0(a^*)\}$. Finally, 
we will assume that such a point $a^*$ is unique and that 
\be\label{e2-2}
\left.\frac{\partial}{\partial\beta}\varphi(\a,\beta)\right|_{(\a,\beta) = a^*} ~<~ 0. 
\ee
Then clearly, $\varphi(a^*) ~=~\varphi_0(a^*)$
and by implicit function theorem, in a neighborhood the point $a^*$, one can parametrize the intersection of the
surfaces ${\cal
  C} ~=~ \{(\a,\beta,t)\in\R^3 : ~t = \varphi(\a,\beta)\}$ and ${\cal
  C}_0 ~=~ \{(\a,\beta,t)\in\R^3 : ~t = \varphi_0(\a,\beta)\}$ as follows : there are
$\eps_1>0$, 
$\eps_2 >0$ and a smooth  function
$\a\to\beta(\a)$ from $[\a^* - \eps_1, \a^* +\eps_2]$ to $\R$ such that $
\beta(\a^*) = \beta^*$ 
and for any $\a^* - \eps_1 \leq \a \leq \a^* + \eps_1$,  
\be
\left.\frac{\partial}{\partial\beta}\varphi(\a,\beta)\right|_{\beta = \beta(\a)} ~<~ 0,
\quad  \quad
\left.\frac{\partial}{\partial\beta}\varphi_0(\a,\beta)\right|_{\beta = \beta(\a)} ~>~ 0, 
\ee
and 
\begin{multline}
\{(\a,\beta,t) \in{\cal C}\cap{\cal C}_0 : \a^* - \eps_1 \leq \a \leq \a^* + \eps_2 \} \\~=~
\{(\a,\beta(\a),t(\a)), \; \a^* - \eps_1 \leq \a \leq \a^* + \eps_2\}
\end{multline}
with 
\be
t(\a) ~=~ \varphi(\a,\beta(\a)) ~=~ \varphi_0(\a,\beta(\a)) ~\geq~ t(\a^*). 
\ee 
Moreover, since the point $a^*$, where the minimum of the function
$\max\{\varphi,\varphi_0\}$ is achieved, is assumed to be unique, the last inequality holds
with the equality if and only if $\a=\a^*$ and without any restriction of generality we
can assume that $t(\a^* - \eps_1) = t(\a^* + \eps_2) ~>~ t(\a^*)$. Then for any $t(\a^*) <
t \leq t(\a^* - \eps_1)$, there are exactly two points $\a^*-\eps_1 \leq \a_1(t) < \a^*$ and
$\a^* < \a_2(t) \leq \a^* + \eps_2$ such that for $a_i(t) = (\a_i(t),\beta(\a_i(t)))$, 
\[
\varphi(a_i(t)) ~=~ \varphi_0(a_i(t)) ~=~ t(\a_i(t) ) ~=~
t, \quad \forall \, i\in\{1,2\},
\]
\[
\Theta^t ~=~ [\a_1(t),\a_2(t)], \quad \hat{D}^t ~=~ \{(\a,\beta) \in \R^2 :
~\varphi(\a,\beta) \leq t, \; \a_1(t) \leq \a \leq a_2(t)\}, 
\]
and $
\Gamma_+^t ~=~ \{a=(\a,\beta)\in\partial_+ D^t : \a_1(t) \leq \a \leq \a_2(t)\}$ 
is the arc on the boundary $\partial_+ D^t$ with the end points in $\tilde{a}_1(t)$ and
$\tilde{a}_2(t)$ where $\tilde{a}_i(t)=(\tilde\a_i(t),\tilde\beta_i(t))$ is a unique point on the
boundary $\partial_+ D^t$ with $\tilde\a_i(t) = \a_i(t)$ for $i=1,2$, 
(see Figure~1).
\begin{figure}[ht]
\resizebox{10cm}{6cm}{\includegraphics{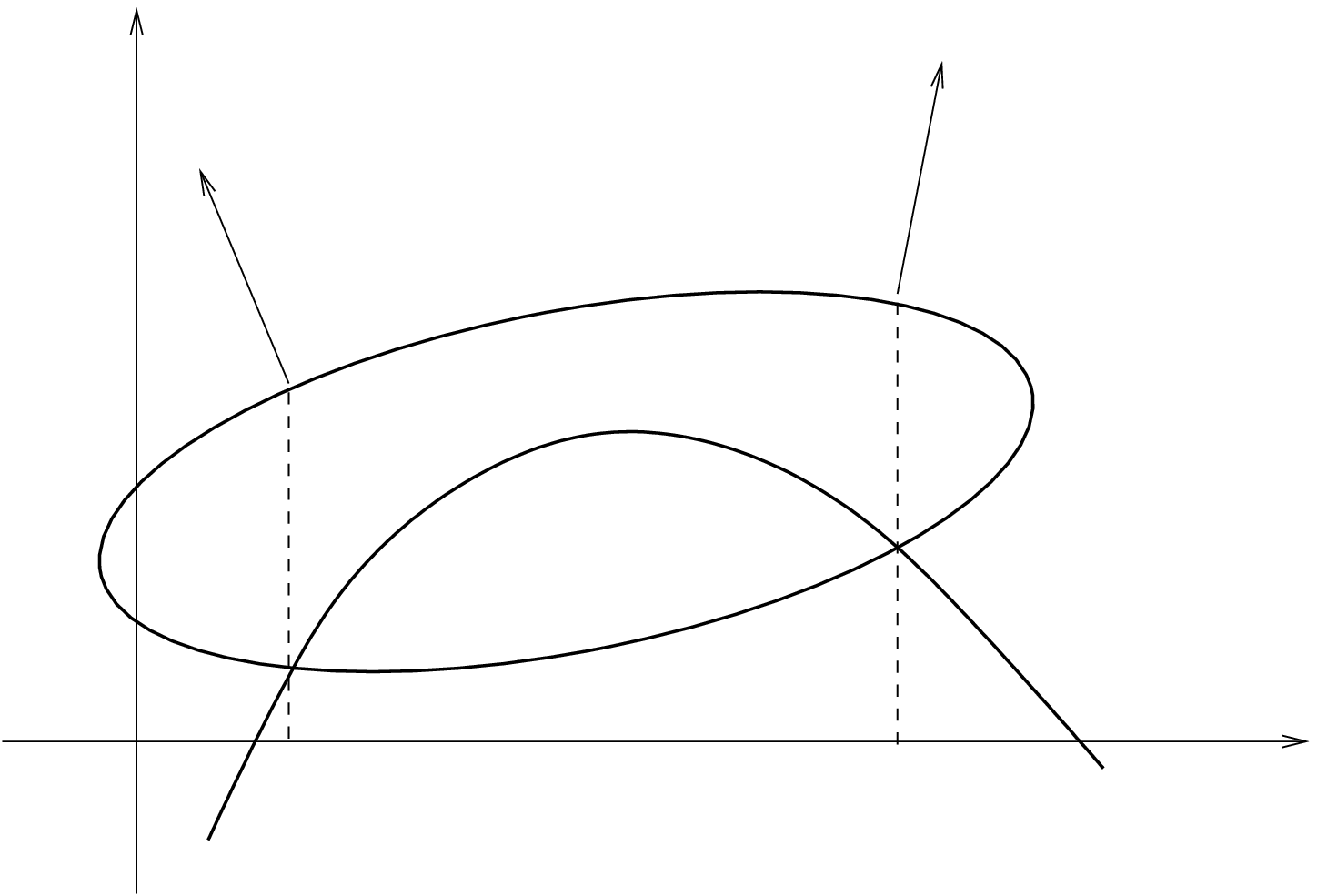}}
\put(-130,40){$D_0^t$}
\put(-88,115){$\tilde{a}_2(t)$}
\put(-83,63){${a}_2(t)$}
\put(-80,135){$\nabla\varphi(\tilde{a}_2(t))$} 
\put(-145,120){$\Gamma_+^t$} 
\put(-91,20){$(\a_2(t),0)$}
\put(-221,20){$(\a_1(t),0)$}
\put(-220,89){$\tilde{a}_1(t)$}
\put(-215,48){${a}_1(t)$}
\put(-230,120){$\nabla\varphi(\tilde{a}_1(t))$} 
\put(-280,20){$(0,0)$} 
\put(-245,60){$D^t$}
\caption{ }
\end{figure}  

\noindent 
Furthermore, by Lemma~\ref{lem2}, for any $t(\a^*) < t \leq t(\a^* - \eps_1)$ and $a\in\Gamma_+^t$, 
\[
V_t(a) ~=~ \begin{cases} \{c_1e_1 + c_2 \nabla\varphi(\tilde{a}_2(t)) : ~ c_i\geq 0\}
  &\text{ if $a=\tilde{a}_2(t)$},\\
 \{- c_1e_1 + c_2 \nabla\varphi(\tilde{a}_1(t)) : ~ c_i\geq 0\} 
  &\text{ if $a=\tilde{a}_1(t)$},\\
\{ c\nabla\varphi(a) : ~c\geq 0\} &\text{ otherwise}.
\end{cases}
\]
Hence, by Theorem~\ref{th1} any sequence of points $z_n\in\Z\times\N$ with
$\lim_n|z_n| = \infty$ converges in the $t$-Martin compactification of $\Z\times\N$ if and
only if one of the following conditions is satisfied : 
\begin{itemize}
\item[--] either $\lim_{n\to\infty} \arg(z_n) = \gamma$ for some 
    $
    \arg(\nabla\varphi(\tilde{a}_2(t)) < \gamma < \arg(\nabla\varphi(\tilde{a}_2(t)),$
\item[--] or $\limsup_{n\to\infty} \arg(z_n) ~\leq~ \arg(\nabla\varphi(\tilde{a}_2(t))$, 
\item[--] or $\liminf_{n\to\infty} \arg(z_n) ~\geq~ \arg(\nabla\varphi(\tilde{a}_2(t))$.
\end{itemize}
In particular, any sequence $z_n\in\Z\times\N$ with
$\lim_n|z_n| = \infty$ and satisfying the inequality $
\arg(z_n) ~\leq~ \arg(\nabla\varphi(\tilde{a}_2(t))$, for all $n\in\N$, 
converges to a point of the $t$-Martin boundary of $\Z\times\N$.  

Remark finally that $a_i(t) \to a^*$ as $t\to t(\a^*)$ for any $i\in\{1,2\}$. From this it
follows that $\tilde{a}_i(t) \to \tilde{a}^*$ as $t\to t(\a^*)$ for any $i\in\{1,2\}$ where
$\tilde{a^*}=(\tilde\a^*,\tilde\beta^*)$ is the unique point on the boundary $\partial_+
D^t$ with $\tilde\a^* = \a^*$, and consequently, 
\[
\lim_{t\to t(a^*)} \nabla\varphi(\tilde{a}_1(t)) ~=~ \lim_{t\to t(a^*)} \nabla\varphi(\tilde{a}_2(t)) ~=~ \nabla\varphi(\tilde{a}^*).
\]
Since clearly, $\nabla\varphi(\tilde{a}_1(t)) \not= \nabla\varphi(\tilde{a}_2(t))$ for
$t(\a^*) < t \leq t(\a^* - \eps_1)$, we conclude that at least one of the function $t \to
\nabla\varphi(\tilde{a}_1(t))$ or $t\to \nabla\varphi(\tilde{a}_2(t))$ is not constant on
the interval $[t(\a^*), t(\a^* - \eps_1)]$ and hence, there are $t,t'\in ]t(\a^*), t(\a^*
  - \eps_1)]$ such that $t\not= t'$ and $
\nabla\varphi(\tilde{a}_i(t)) ~\not=~ \nabla\varphi(\tilde{a}_i(t'))$ 
either for $i=1$ or for $i=2$. Suppose that this relation holds for $i=2$ (the case when
$i=1$ is quite similar) and let 
\[
\arg(\nabla\varphi(\tilde{a}_i(t))) ~<~ \arg(\nabla\varphi(\tilde{a}_i(t'))).
\]
Then in the $t'$-Martin compactification, any sequence of points $z_n\in\Z\times\N$ with
$\lim_n|z_n| = \infty$ and 
\[
\arg(\nabla\varphi(\tilde{a}_i(t))) ~\leq~ \arg(z_n) ~\leq~
\arg(\nabla\varphi(\tilde{a}_i(t'))), \quad \forall n\in\N,
\]
converges to a point of the $t'$-Martin boundary, while in
the $t$-Martin compactification 
such a sequence converges to a point of the $t$-Martin boundary if and only if there exist a
limit $\lim_n z_n/|z_n|$. The following proposition is therefore proved.

\begin{prop}\label{pr2-1} Let the conditions (H0)-(H4) be satisfied. Suppose moreover that the minimum
  of the function $\max\{\varphi,\varphi_0\}$ is attained at a 
  unique point $a^*$ and the inequalities \eqref{e2-1} and \eqref{e2-2} hold. Then the $t$-Martin
  compactification of the transition kernel $P$ is unstable.
\end{prop}

\section{Proof of Proposition~\ref{pr1}}\label{proof-of-prop1}
We prove this proposition by using large deviation principle of the sample paths of scaled
processes $Z^\eps(t) = \eps Z([t/\eps])$ with $\eps \to 0$. Before proving this proposition we
recall the definition of the sample path large deviation principle. 

Throughout this section, for $t\in[0,+\infty[$, we denote by $[t]$ the integer part of
    $t$. 

\smallskip
\noindent
{\bf Definitions : \;}{\em 1) Let $D([0,T],\R^{d})$ denote the set of all right continuous with left
limits functions from $[0,T]$ to $\R^{d}$ endowed with Skorohod metric
(see Billingsley~\cite{Billingsley}). Recall that   
a mapping $I_{[0,T]}:~D([0,T],\R^{d})\to
[0,+\infty]$ is a good rate function on $D([0,T],\R^{d})$ if for any $c\geq 0$ and 
 any compact set $V\subset \R^{d}$, the set
\[
\{ \varphi \in D([0,T],\R^{d}): ~\phi(0)\in V \; \mbox{
and } \; I_{[0,T]}(\varphi) \leq c \}
\]
is compact in $D([0,T],\R^{d})$.  According to this definition, a good
rate function is lower semi-continuous. 

2) For a Markov chain $(Z(t))$ on $E\subset\R^d$ the family of scaled processes
$(Z^\eps(t) =\eps Z([t/\eps]), 
\,t\in[0,T])$,  is said to
satisfy {\it sample path large deviation principle} in $D([0,T], \R^{d})$ with a rate function
$I_{[0,T]}$ if for any $z\in\R^{d}$ 
\begin{equation}\label{e4-1}
\lim_{\delta\to 0} \;\liminf_{\eps\to 0} \; \inf_{z'\in E : |\eps z'-z|<\delta} \eps
\log\P_{z'}\left( Z^\eps(\cdot)\in {\cal 
O}\right) \geq -\inf_{\phi\in{\cal O}:\phi(0)=z} I_{[0,T]}(\phi), 
\end{equation}
for every open set ${\cal
O}\subset D([0,T],\R^{d})$,
and
\begin{equation}\label{e4-2}
\lim_{\delta\to 0} \;\limsup_{\eps\to 0} \; \sup_{z' \in E : |\eps z'-z|<\delta}
\eps\log\P_{z'}\left( Z^\eps(\cdot)\in 
F\right) \leq -\inf_{\phi\in F:\phi(0)=z} I_{[0,T]}(\phi).
\end{equation}
 for every closed set $F\subset   D([0,T],\R^{d})$. 
}

We refer to sample path large deviation
principle as SPLD principle. Inequalities (\ref{e4-1}) and (\ref{e4-2}) are
referred as lower and upper SPLD bounds respectively.

Recall that the convex conjugate $f^*$ of a function $f :\R^d\to \R$ is defined by 
\[
f^*(v) ~=~ \sup_{a\in\R^d} (a\cdot v - f(a)), \quad v\in\R^d. 
\]

The following 
proposition provides the SPLD principle for the  scaled processes
$Z^\eps(t) ~=~ \eps Z([t/\eps])$ for our random walk $(Z(n))$ on $\Z\times\N$. 
\begin{prop}\label{pr4-1} Under  the hypotheses $(H_0)-(H_4)$, for every $T>0$,  the family of  scaled  
  processes $(Z^\eps(t), \, t\in[0,T])$ satisfies SPLD
  principle in $D([0,T], \R^{d})$  with a 
  good rate function 
\[
I_{[0,T]}(\phi) ~=~ \begin{cases} \int_0^T L(\phi(t),\dot\phi(t)) \, dt, &\text{ if
    $\phi$ is absolutely continuous and  }\\ &\text{ $\phi(t)\in\R^{d-1}\times\R_+$ for all
    $t\in[0,T]$,}\\
+\infty &\text{ otherwise.}
\end{cases}
\]
where  for any  $z=(x,y)\R^{d-1}\times[0,+\infty[$
    and $v\in\R^{d}$, the local rate function $L$ is given by 
\[
L(z,v) ~=~ \begin{cases} (\log \varphi)^*(v) &\text{ if $y>0$},\\
(\log~\max\{\varphi, \varphi_0\})^*(v) &\text{ if $y=0$}.
\end{cases}
\]
\end{prop}
\noindent

This proposition is a consequence of the results obtained in
\cite{D-E-W,D-E,Ignatiouk:02,Ignatiouk:04}. The results of Dupuis, Ellis and
Weiss~\cite{D-E-W} prove that $I_{[0,T]}$ is a good rate function on $D([0,T],\R^d)$
and provide the SPLD upper bound. SPLD lower
bound follows from the local estimates 
obtained in \cite{Ignatiouk:02}, the general SPLD lower bound of Dupuis and
Ellis~\cite{D-E} and the integral representation of the corresponding rate function
obtained in \cite{Ignatiouk:04}. 

We are ready now to complete the proof of Proposition~\ref{pr1}. The proof of the upper bound
\be\label{e4-3}
\rho(P) ~\leq~ \inf_{a\in\R^d}\max\{\varphi(a), \varphi_0(a)\}
\ee
is quite simple. Recall that $\rho(P)$ is equal to the infimum of 
all those $t>0$ for which the inequality $P f \leq t f$ has a non-zero solution
$f>0$, see Seneta~\cite{Seneta}. Since for any $a\in\R^d$, this inequality is satisfied
with $t=\max\{\varphi(a),\varphi_0(a)\}$ for an exponential function
$f(z) ~=~ \exp(a\cdot z)$, one gets therefore $\rho(P) ~\leq~ \max\{\varphi(a), \varphi_0(a)\}$ for all $a\in\R^d$, and
consequently, \eqref{e4-3} holds. To prove   the lower bound 
\be\label{e4-4}
\rho(P) ~\geq~ \inf_{a\in\R^d}\max\{\varphi(a), \varphi_0(a)\}
\ee
we use  the results of the paper ~\cite{Ignatiouk:05}. Theorem~1 of
~\cite{Ignatiouk:05} proves that for a zero constant function $\ol{0}(t) = 0$,
$t\in[0,T]$, 
\[
\log\rho(P) ~=~ - \frac{1}{T} I_{[0,T]}(\ol{0}) 
\]
whenever the following conditions
are satisfied : 
\begin{enumerate} 
\item[$(a_1)$] for every $T>0$, the family of
rescaled processes $(Z_\eps(t), \; t\in[0,T])$ satisfies sample
path large deviation principle in $D([0,T], \R^{d-1}\times[0,\infty[)$ with a good rate functions
$I_{[0,T]}$;   
\item[$(a_2)$] the rate function $I_{[0,T]}$ has an integral form~: there is a local rate function
$L:(\R^{d-1}\times[0,\infty[)\times\R^d\to\R_+$ such that 
\[
I_{[0,T]}(\phi) = \int_0^T L(\phi(t),\dot\phi(t))  \,dt 
\]
if the function $\phi : [0,1]\to \R^{d-1}\times[0,\infty[$ is absolutely continuous, and $I_{[0,1]}(\phi)
= +\infty$ otherwise.
\item[$(a_3)$] there are two convex functions $l_1$ and $l_2$
  on $\R^d$   such that  
\begin{itemize}
\item[--]   $0\leq l_1(v) \leq L(x,v) \leq l_2(v)$ for all $x\in\R^{d-1}\times[0,\infty[$ and $v\in\R^d$, 
\item[--] the function $l_2$ is finite in a neighborhood of zero 
\item[--] and $$
\lim_{n\to\infty} \inf_{|v|\geq n} l_1(v)/|v| > 0.$$
\end{itemize}
\end{enumerate}
In our setting, the  conditions $(a_1)$ and $(a_2)$ are satisfied by
Proposition~\ref{pr4-1} and the condition $(a_3)$ is satisfied with 
$l_1(v) ~=~ (\log(\varphi,\varphi_0))^*(v)$ and $l_2(v) ~=~ (\log\varphi)^*(v)$~:
 
\begin{itemize}
\item[--]  Clearly, $
(\log(\varphi,\varphi_0))^*(v) ~\leq~ L(x,v) ~\leq~ (\log\varphi)^*(v)$
for all $x\in\R^{d-1}\times[0,\infty[$ and $v\in\R^d$. 
\item[--] Under the hypotheses (H2) and (H3), there is $\delta > 0$ such that 
\[
\liminf_{|a|\to\infty} \frac{1}{|a|}\log\varphi(a) ~>~ \delta,
\]
and consequently, 
\begin{align*}
\sup_{v\in\R^d : |v| \leq \delta} (\log\varphi)^*(v) &~=~ \sup_{v\in\R^d : |v| \leq \delta}
~\sup_{a\in\R^d} \Bigl(a\cdot v - \log\varphi(a)\Bigr)\\
&~=~ \sup_{a\in\R^d} ~\sup_{v\in\R^d : |v| \leq \delta}\Bigl(a\cdot v - \log\varphi(a)\Bigr)\\
&~=~  \sup_{a\in\R^d} ~\Bigl(\delta |a| - \log\varphi(a)\Bigr) ~<~ +\infty. 
\end{align*} 
The function $(\log\varphi)^*(v)$ is therefore finite in a neighborhood of zero. 
\item[--] For any $r >0$, 
\begin{align*}
(\log(\varphi,\varphi_0))^*(v) &~\geq~ \sup_{a\in\R^d: ~|a|\leq r} \Bigl(a\cdot v -
  \log(\varphi,\varphi_0)(a) \Bigr)\\
&~\geq~ \sup_{a\in\R^d: ~|a|\leq r} a\cdot v ~-
  \sup_{a\in\R^d: ~|a|\leq r} \log(\varphi,\varphi_0)(a) \\
&~\geq~ r |v| ~-
  \sup_{a\in\R^d: ~|a|\leq r} \log(\varphi,\varphi_0)(a).
\end{align*}
Since by (H3), the function $\log(\varphi,\varphi_0)$ is finite
  everywhere on $\R^d$, from this it follows that 
\[
\lim_{n\to\infty} ~\inf_{|v|\geq n} ~\frac{1}{|v|}(\log(\varphi,\varphi_0))^*(v) ~\geq~ r ~>~
0. 
\]
\end{itemize}
Using Theorem~1 of~\cite{Ignatiouk:05} and the explicit form of the local rate
function $L$ one gets 
\begin{align*}
\log\rho(P) ~=~ - \frac{1}{T} I_{[0,T]}(\ol{0}) &~=~ - L(0,0) ~= - (\log~\max\{\varphi, \varphi_0\})^*(0) \\&~=~ \log
\inf_{a\in\R^d}\max\{\varphi(a), \varphi_0(a)\}.
\end{align*}
Proposition~\ref{pr1} is therefore proved.

\section{Proof of Theorem~\ref{th1}}\label{proof-of-theorem}
In a particular case, for $t=1$, this theorem was proved in ~\cite{Ignatiouk:07} under
slight different conditions : in addition to the hypotheses (H0)-(H4), the positive
measures $\mu$ and $\mu_0$ were assumed to be probability measures and 
the means 
\[
m ~\dot=~ \sum_{z\in\Z^d} \mu(z) \, z \quad \text{ and } \quad m_0~=~ \sum_{z\in\Z^d}
\mu_0(z)\, z
\]
were assumed to satisfy the following condition :
\be\label{e5-1}
m/|m| + m_0/|m_0| ~\not=~ 0. 
\ee
Remark that under the above assumptions, the set $\partial D^1\cap \partial D_0^1$
contains the point zero and the set $D^1\cap D^1_0$ has a non-empty interior. By
Proposition~\ref{pr1} from this it follows that 
\be\label{e5-2}
\rho(P) ~=~ \inf_{a\in\R^d} \max\{\varphi(a),\varphi_0(a)\} ~<~ 1.
\ee
The above additional conditions can be replaced by a weaker one : for $t=1$, with the same
arguments as in ~\cite{Ignatiouk:07} one can get Theorem~\ref{th1}  when $\mu$ is a
probability measure on $\Z^d$ and $\mu_0$ is a positive measure on $\Z^d$ satisfying the
inequality \eqref{e5-2} such that $\mu_0(\Z^d)\leq 1$.  
This result is now combined with  the exponential change of the
measure in order to prove Theorem~\ref{th1} for 
\[
t ~>~ \rho(P) ~=~ \inf_{a\in\R^d} \max\{\varphi(a),\varphi_0(a)\}. 
\]
For any $t$ satisfying this inequality, there is a point $\tilde{a}_t\in\partial D^t\cap
D_0^t$. We consider a twisted random walk
$(\tilde{Z}(t))$ on $\Z^{d-1}\times\N$ with transition probabilities 
\[
\tilde{p}(z,z') ~=~ \begin{cases} \mu(z'-z) \exp(\tilde{a}_t\cdot(z'-z))/t &\text{ if
    $z=(x,y)\in\Z^{d-1}\times\N$ with $y>0$,}\\
\mu_0(z'-z) \exp(\tilde{a}_t\cdot(z'-z))/t &\text{ if
    $z=(x,y)\in\Z^{d-1}\times\N$ with $y=0$.}
\end{cases}
\]
For such a random walk $(\tilde{Z}(n))$, the jump
generating functions are given by 
\[
\tilde\varphi(a) ~=~ \sum_{z\in\Z^d} \tilde\mu(z) \exp(a\cdot z) ~=~
\varphi(a+\tilde{a}_t)/t,
\] 
and 
\[ 
\quad \tilde\varphi_0(a) ~=~ \sum_{z\in\Z^d} \tilde\mu_0(z) \exp(a\cdot z) ~=~ \varphi_0(a+\tilde{a}_t)/t.
\] 
Hence, 
\[
\tilde{D}^1 ~\dot=~ \{ a\in\R^d : ~\tilde\varphi(a) \leq 1\} ~=~ \{a \in\R^d : \varphi(a +
\tilde{a}_t) \leq t \} ~=~ -
\tilde{a}_t + D^t, 
\]
and similarly,
\[
\tilde{D}_0^1 ~\dot=~ \{ a\in\R^d : ~\tilde\varphi_0(a) \leq 1\} ~=~ -\tilde{a}_t + D^t_0.
\]
Moreover, with the same arguments one gets 
\[
\tilde{\Theta}^1 ~\dot=~ \bigl\{\a\in\R^{d-1} : \inf_{\beta\in\R}\max\{ \tilde\varphi(a),
\tilde\varphi_0(a) \} \leq 1\bigr\} ~=~ - \tilde\a_t + \Theta^t 
\]
where $\a_t$ denotes the vector of $d-1$ first coordinates of $\tilde{a}_t$,  
\[
\hat{\tilde{D}}^1 ~\dot=~ (\tilde{\Theta}^1\times\R)\cap \tilde{D}^1 ~=~ -\tilde{a}_t +
\hat{D}^t \quad \text{ and } \quad \tilde\Gamma_+^1 ~\dot=~
\hat{\tilde{D}}^1\cap\partial_+\tilde{D}^1 ~=~ -\tilde{a}_t + \Gamma^t_+.
\]
For any $a\in \Gamma_+^t$ the normal cone $V_t(a)$ to
the set $\hat{D}^t$ at the point $a$ is therefore identical to the normal cone 
$\tilde{V}_1(a-\tilde{a}_t)$ to the set $\hat{\tilde{D}}^1$  at the point $a-\tilde{a}_t\in\tilde\Gamma_+^1$. Remark finally 
that for any $a\in\tilde\Gamma_+^1$ the functions $\tilde{h}_{a,1}$ defined by \eqref{e1-13} 
with  $t=1$ and the functions $\tilde\varphi$ and $\tilde\varphi_0$ instead of $\varphi$ and
$\varphi_0$, satisfy the equality 
\[
\tilde{h}_{a,1}(z)(a) ~=~ h_{a+\tilde{a},t}(z) \exp(-\tilde{a}_t\cdot z), \quad \forall
z\in\Z^{d-1}\times\N. 
\]
Since clearly,
\begin{align*}
\tilde{G}_1(z,z') &~\dot=~ \sum_{n=0}^\infty \P_z(\tilde{Z}(n) = z') ~=~ \sum_{n=0}^\infty t^{-n}
  \P_z(Z(n) = z') \exp(\tilde{a}\cdot (z'-z)) \\&~=~ G_t(z,z') \exp(\tilde{a}\cdot
  (z'-z)), \quad \forall z,z'\in\Z^{d-1}\times\N,
\end{align*}
we conclude therefore that 
\begin{itemize}
\item[(i)] for any unit vector $q\in\R^{d-1}\times[0,+\infty[$  there
    exists a unique point $\hat{a}_t(q)\in \Gamma_+^t$  such that $q\in
    V_t(\hat{a}_t(q))$, 
\item[(ii)] for any
  $a\in\hat{D}^t\cap\partial_+ D^t$ and any sequence of points
  $z_n\in\Z^{d-1}\times\N$, 
\begin{align*}
\lim_{n\to\infty} K_t(z,z_n) &=~ \lim_{n\to\infty} G_t(z,z_n)/G_t(z_0,z_n) \\&=~ \exp(\tilde{a}\cdot
  (z-z_0))  ~\lim_{n\to\infty}  \tilde{G}_1(z,z_n)/\tilde{G}_1(z_0,z_n)  \\&=~ \exp(\tilde{a}\cdot
  (z-z_0)) ~\tilde{h}_{a-\tilde{a}_t,t}(z)/\tilde{h}_{a-\tilde{a}_t,t}(z_0) 
\\&=~ h_{a,t}(z)/h_{a,t}(z_0),\quad \quad \quad \quad \forall \;
z\in\Z^{d-1}\times\N, 
\end{align*}
whenever $\lim_{n\to\infty} |z_n|=\infty$ and $\lim_{n\to\infty}
  \dist(V_t(a),z_n/|z_n|)=0$.
\end{itemize}
Theorem~\ref{th1} is therefore proved. 

\providecommand{\bysame}{\leavevmode\hbox to3em{\hrulefill}\thinspace}
\providecommand{\MR}{\relax\ifhmode\unskip\space\fi MR }
\providecommand{\MRhref}[2]{%
  \href{http://www.ams.org/mathscinet-getitem?mr=#1}{#2}
}
\providecommand{\href}[2]{#2}

\end{document}